\documentclass[12pt,a4paper]{article}
\usepackage{latexsym,amsfonts,amssymb,amsmath,amscd,theorem,epic,theorem,epstopdf} 
\theoremstyle{plain}
\newtheorem{lemma}{Lemma}[section]
\newtheorem{proposition}[lemma]{Proposition}

\newtheorem{remark}[lemma]{Remark}
\newtheorem{example}[lemma]{Example}
\newtheorem{theorem}[lemma]{Theorem}
\newtheorem{definition}[lemma]{Definition}
\newtheorem{corollary}[lemma]{Corollary}

{\theorembodyfont{\rmfamily}  \font\ncsc=cmcsc10
 \font\ntt=cmtt12

\begin{document}
\newcommand{\val}{\operatorname{val}}\newcommand{\Val}{\operatorname{Val}}
\newcommand{\pperp}{\hbox{$\perp\hskip-6pt\perp$}}
\newcommand{\ssim}{\hbox{$\hskip-2pt\sim$}}\newcommand{\ini}{\operatorname{ini}}
\newcommand{\aleq}{{\ \stackrel{3}{\le}\ }}\newcommand{\Vers}{\mathrm{Vers}}
\newcommand{\ageq}{{\ \stackrel{3}{\ge}\ }}\newcommand{\Tub}{\operatorname{Tub}}
\newcommand{\aeq}{{\ \stackrel{3}{=}\ }}
\newcommand{\bleq}{{\ \stackrel{n}{\le}\ }}
\newcommand{\bgeq}{{\ \stackrel{n}{\ge}\ }}
\newcommand{\beq}{{\ \stackrel{n}{=}\ }}
\newcommand{\cleq}{{\ \stackrel{2}{\le}\ }}
\newcommand{\cgeq}{{\ \stackrel{2}{\ge}\ }}
\newcommand{\ceq}{{\ \stackrel{2}{=}\ }}
\newcommand{\fm}{\mathfrak{m}}\newcommand{\bi}{\boldsymbol{i}}
\newcommand{\N}{{\mathbb N}}\newcommand{\T}{{\mathbb T}}
\newcommand{\A}{{\mathbb A}}\newcommand{\Arg}{\operatorname{Arg}}
\newcommand{\K}{{\mathbb K}}\newcommand{\Rea}{\operatorname{Re}}
\newcommand{\Z}{{\mathbb Z}}\newcommand{\F}{{\mathbf F}}
\newcommand{\R}{{\mathbb R}}
\newcommand{\C}{{\mathbb C}}
\newcommand{\Q}{{\mathbb Q}}
\newcommand{\PP}{{\mathbb P}}
\newcommand{\cA}{{\mathcal A}}
\newcommand{\cB}{{\mathcal B}}
\newcommand{\cC}{{\mathcal C}}
\newcommand{\cD}{{\mathcal D}}
\newcommand{\cF}{{\mathcal F}}
\newcommand{\cI}{{\mathcal I}}
\newcommand{\cL}{{\mathcal L}}
\newcommand{\cM}{{\mathcal M}}
\newcommand{\cR}{{\mathcal R}}
\newcommand{\cO}{{\mathcal O}}
\newcommand{\cP}{{\mathcal P}}
\newcommand{\cS}{{\Sigma}}
\newcommand{\cT}{{\mathfrak T}}
\newcommand{\mcT}{{\mathcal T}}
\newcommand{\cU}{{\mathcal U}}
\newcommand{\cZ}{{\mathcal Z}}
\newcommand{\cOK}{\mathcal{OK}}
\newcommand{\cEA}{{\mathcal EA}}
\newcommand{\cNDF}{\mathcal{NDF}}
\newcommand{\mkfO}{{\mathfrak O}}
\newcommand{\go}{O}
\newcommand{\gc}{{\mathfrak c}}
\newcommand{\newDelta}{{E}}
\newcommand{\bbz}{{x_0}}
\newcommand{\Al}{\operatorname{Ad}}\newcommand{\Men}{\operatorname{Men}}
\newcommand{\EAl}{\operatorname{EAd}}\newcommand{\Ann}{\operatorname{Ann}}
\newcommand{\red}{{\operatorname{red}}}
\newcommand{\Arn}{{\operatorname{AR}}}
\newcommand{\Pic}{{\operatorname{Pic}}}\newcommand{\Sym}{{\operatorname{Sym}}}
\newcommand{\QI}{{\operatorname{QI}}}\newcommand{\Div}{{\operatorname{Div}}}
\newcommand{\oDel}{{\widetilde\Del}}
\newcommand{\real}{{\operatorname{Re}}}\newcommand{\Aut}{{\operatorname{Aut}}}
\newcommand{\conv}{{\operatorname{conv}}}\newcommand{\Ima}{{\operatorname{Im}}}
\newcommand{\Span}{{\operatorname{Span}}}\newcommand{\Trop}{{\operatorname{Trop}}}
\newcommand{\Ker}{{\operatorname{Ker}}}\newcommand{\qind}{{\operatorname{qind}}}
\newcommand{\Cycle}{{\operatorname{Cycle}}}\newcommand{\OG}{{\operatorname{OG}}}
\newcommand{\Fix}{{\operatorname{Fix}}}\newcommand{\ina}{{\operatorname{in}}}
\newcommand{\sign}{{\operatorname{sign}}}\newcommand{\Closure}{{\operatorname{Clos}}}
\newcommand{\even}{{\operatorname{even}}}
\newcommand{\odd}{{\operatorname{odd}}}
\newcommand{\com}{{\operatorname{com}}}
\newcommand{\ncom}{{\operatorname{ncom}}}
\newcommand{\nmob}{{\operatorname{nmob}}}
\newcommand{\bound}{{\operatorname{bound}}}
\newcommand{\nbound}{{\operatorname{ends}}}
\newcommand{\Inn}{{\operatorname{In}}}
\newcommand{\Ex}{{\operatorname{Ex}}}
\newcommand{\opp}{{\operatorname{opp}}}\newcommand{\Par}{{\operatorname{Par}}}
\newcommand{\Cheb}{{\operatorname{Cheb}}}\newcommand{\arccosh}{{\operatorname{arccosh}}}
\newcommand{\Card}{{\operatorname{Card}}}
\newcommand{\alg}{{\operatorname{alg}}}
\newcommand{\ord}{{\operatorname{ord}}}
\newcommand{\mt}{{\operatorname{mult}}}
\newcommand{\cheb}{{\operatorname{Cheb}}}
\newcommand{\oi}{{\overline i}}\newcommand{\oGamma}{{\overline\Gamma}}
\newcommand{\oj}{{\overline j}}\newcommand{\oh}{{\overline h}}
\newcommand{\ob}{{\overline b}}
\newcommand{\os}{{\overline s}}
\newcommand{\oa}{{\overline a}}
\newcommand{\oy}{{\overline y}}
\newcommand{\ow}{{\overline w}}
\newcommand{\ot}{{\overline t}}
\newcommand{\oz}{{\overline z}}
\newcommand{\eps}{{\varepsilon}}
\newcommand{\proofend}{\hfill$\Box$\bigskip}
\newcommand{\Int}{{\operatorname{Int}}}
\newcommand{\pr}{{\operatorname{pr}}}
\newcommand{\Hom}{{\operatorname{Hom}}}
\newcommand{\Ev}{{\operatorname{Ev}}}
\newcommand{\im}{{\operatorname{Im}}}\newcommand{\br}{{\operatorname{br}}}
\newcommand{\sk}{{\operatorname{sk}}}\newcommand{\DP}{{\operatorname{DP}}}
\newcommand{\const}{{\operatorname{const}}}
\newcommand{\Sing}{{\operatorname{Sing}}\hskip0.06cm}
\newcommand{\conj}{{\operatorname{Conj}}}
\newcommand{\Cl}{{\operatorname{Cl}}}
\newcommand{\Crit}{{\operatorname{Crit}}}
\newcommand{\Ch}{{\operatorname{Ch}}}
\newcommand{\discr}{{\operatorname{discr}}}
\newcommand{\Tor}{{\operatorname{Tor}}}
\newcommand{\Conj}{{\operatorname{Conj}}}
\newcommand{\Log}{{\operatorname{Log}}}
\newcommand{\vol}{{\operatorname{vol}}}
\newcommand{\defect}{{\operatorname{def}}}
\newcommand{\codim}{{\operatorname{codim}}}
\newcommand{\tmu}{{\C\mu}}
\newcommand{\wt}{{\operatorname{wt}}}
\newcommand{\ov}{{\overline v}}
\newcommand{\ox}{{\overline{x}}}
\newcommand{\bw}{{\boldsymbol w}}
\newcommand{\hbw}{{\widehat\bw}}
\newcommand{\mfw}{{\mathfrak{w}}}
\newcommand{\bv}{{\boldsymbol v}}
\newcommand{\bn}{{\Phi}}
\newcommand{\bx}{{\boldsymbol x}}
\newcommand{\bd}{{\boldsymbol d}}
\newcommand{\bz}{{\boldsymbol z}}
\newcommand{\bL}{{\boldsymbol L}}
\newcommand{\bP}{{\boldsymbol P}}
\newcommand{\bp}{{\boldsymbol p}}
\newcommand{\bq}{{\boldsymbol p_{tr}}}
\newcommand{\be}{{\boldsymbol e}}
\newcommand{\bc}{{\boldsymbol c}}
\newcommand{\ba}{{\boldsymbol a}}
\newcommand{\bb}{{\boldsymbol b}}
\newcommand{\tet}{{\theta}}
\newcommand{\Area}{\operatorname{Area}}
\newcommand{\Del}{{\Delta}}
\newcommand{\bet}{{\beta}}
\newcommand{\kap}{{\kappa}}
\newcommand{\del}{{\delta}}
\newcommand{\sig}{{\sigma}}
\newcommand{\alp}{{\alpha}}
\newcommand{\Sig}{{\Sigma}}
\newcommand{\Gam}{{\Gamma}}\newcommand{\Rot}{{\operatorname{Rot}}}
\newcommand{\gam}{{\gamma}}\newcommand{\idim}{{\operatorname{idim}}}
\newcommand{\Lam}{{\Lambda}}
\newcommand{\lam}{{\lambda}}
\newcommand{\SC}{{SC}}
\newcommand{\MC}{{MC}}
\newcommand{\nek}{{,...,}}
\newcommand{\cim}{{c_{\mbox{\rm im}}}}
\newcommand{\clM}{\tilde{M}}
\newcommand{\clV}{\bar{V}}
\newcommand{\rtm}{{\mu}}

\title{Quantum index, Arnold-Rokhlin surfaces,\\ and real enumerative geometry}
\author{Ilia Itenberg
\and Eugenii Shustin}
\date{} 
\maketitle
\centerline{\it To the memory of Vladimir Igorevich Arnold, great man and great mathematician}
\begin{abstract}
The 
main goal of this note is to relate two different Welschinger-type rules of signs that were used 
for definition
of real enumerative invariants of toric surfaces
(the invariants considered being relative to the toric boundary). 
The relation between them intertwines Mikhalkin's quantum index and geometry
of real and complex point sets of counted real curves. As a by-product,
we suggest a new definition of quantum index, which can be 
used 
for refined invariant enumeration of real algebraic curves
on surfaces in a non-toric setup.
\end{abstract} 

\medskip

{\bf MSC-2020 classification:} Primary 14N10, Secondary 14J26, 14P25 

\thanks{The first author was supported in part by the ANR grant
ANR-22-CE40-0014 SINTROP. 
The second author has been partly supported by 
the Bauer-Neuman Chair in Real and Complex Geometry.} 

\medskip

\section{Introduction}

The paper is devoted to enumeration of real curves
satisfying certain constraints
on real algebraic surfaces
(by a {\it real algebraic surface}, we mean a complex algebraic one equipped with an anti-holomorphic involution). 
In fact, we enumerate {\it oriented} real curves. 
An irreducible real algebraic curve $\nu:\widetilde C\to C\hookrightarrow X$ (where $\nu$ is the normalization)
on
a real algebraic surface $X$
is said to be {\it separating} if $\widetilde C\setminus\widetilde C_\R$ consists of two connected components. 
The choice of one of these components (denoted by $\widetilde C_+$)
is called
an {\it orientation}
of the separating curve. A pair $(C,C_+)$,
where $C_+=\nu(\widetilde C_+)$, is called an
{\it oriented real curve}.
The orientation induced
on the real part $C_\R$ of $C$ by the canonical orientation of $C_+$
is called {\it complex orientation}.

In his groundbreaking paper \cite{Mir}, G. Mikhalkin
introduced an integer-valued {\it quantum index}
for certain oriented real curves on toric surfaces. More precisely, the quantum index was defined 
by Mikhalkin for oriented real curves that intersect 
toric divisors only at real or purely imaginary points: under this condition, \cite[Theorem 3.1]{Mir} states
that the number ({\it aka} quantum index)
$$\QI(C_+)=\frac{1}{\pi^2}\int_{C_+}\frac{dx \wedge dy}{xy},\quad\text{where}\quad x = |u|,\ y = |v|,\ (u, v)\in(\C^*)^2,$$ 
belongs to $\frac{1}{2}\Z$. 
Mikhalkin showed that, for an appropriate kind of constraints,
a Welschinger-type signed enumeration of
oriented real rational curves ({\it cf}. \cite{W1})
in a given divisor class and with a given quantum index 
produces an invariant
which can be directly related to the numerator
of a Block-G\"ottsche refined tropical invariant (represented as a fraction with the standard denominator,
see \cite{BG}).

In \cite{ISalg}, we extended the invariance part of Mikhalkin's statement to the case
of oriented real curves of genus $g$, where $g$ is equal to $1$ or $2$. In addition, we introduced
a second Welschinger-type rule of signs (simpler than the first one
from several perspectives) also producing an invariant enumeration of oriented real curves
under consideration.

In the present paper, we show that the two rules of signs considered in \cite{ISalg}
coincide (up to some standard factor
depending only on the chosen divisor class and genus) for each individual oriented real curve
under enumeration.
As a by-product, we 
suggest a new definition of quantum index, 
which can be taken into account for enumeration of oriented real curves on real algebraic surfaces in a non-toric setup.
In some cases, this quantum index allows one to define refined real enumerative invariants
beyond the range studied in \cite{Mir,Bl25,ISalg}
(these new refined invariants will be described in a forthcoming paper).

The paper is organized as follows.
Section \ref{sec_theorem-signs} is devoted to the proof of the fact the two rules of signs
described in \cite{ISalg}
coincide.
The proof uses a reformulation of the definition of the quantum index in terms of {\it Arnold-Rokhlin surfaces}.
In Section \ref{sec-another},
we give a definition of the quantum index with the help of Arnold-Rokhlin surfaces
in the toric setting (in this case,
the definition is a reformulation of the definitions of quantum index that appeared in \cite{Mir, Bl25},
see Section \ref{sec-toric} for details) and in a non-toric setup (see Section \ref{sec-nt}).

\vskip5pt

{\bf Acknowledgements}.
We are grateful to Akiyoshi Sannai for attracting our attention to an incorrect description of signs
of tangency points in Section 3 of the first version of the paper.

\section{Theorem on signs}\label{sec_theorem-signs}

\subsection{Main statement}\label{sec-main} 

Let $\Tor(P)$ be the complex toric surface associated with a non-degenerate convex lattice polygon $P\subset\R^2$.
Denote by $\Tor(\partial P)$ the union of the toric divisors and by ${\mathcal L}_P$ the tautological line bundle on $\Tor(P)$.
Put $D_P=c_1({\mathcal L}_P)$; we freely use the notation $|D_P|$ for the corresponding linear system
on $\Tor(P)$.
The real part $\Tor_\R(P)$ consists of four quadrants
$${\mathbf Q}(\eps_1,\eps_2)=\Closure\{(x,y)\in(\R^\times)^2\ :\ \eps_1 x > 0, \ \eps_2 y > 0\},\quad\eps_1,\eps_2=\pm,$$
where $\R^\times = \R \setminus \{0\}$ and $\Closure$ stands for the closure.
Fix the orientation of ${\mathbf Q}(+, +)$ given by the form $dx\wedge dy$,
and fix the orientations of other quadrants induced by reflections
\begin{align}
{\mathbf Q}(+, +) & \to {\mathbf Q}(\eps_1,\eps_2), \nonumber \\ 
(x,y) & \mapsto (\eps_1 x, \eps_2 y) \nonumber 
\end{align} 
(in the same way as in \cite{Mir}).
These orientations are said to be {\it canonical}.
Denote by $P^1$ the set of sides of $P$.

In \cite{ISalg}, we introduced two series of real enumerative invariants 
(the first of these series in the case of rational curves was introduced by Mikhalkin in \cite{Mir}) 
counting pairs $(C,C_+)$, where
\begin{enumerate}
\item[(a)] $C\subset\Tor(P)$,
$C\in|2D_P|$, is a real nodal ({\it i.e.}, the only possible singularities are ordinary nodes) curve 
of genus $g\in\{0,1,2\}$ having one real branch in ${\mathbf Q}(+,+)$ and $g$ more real branches that lie in
distinct non-positive closed quadrants; notice that such a real curve $C$ is necessarily separating (it is even {\it maximal},
that is, it has the maximal possible number of real branches for a real curve of genus $g$); 
\item[(b)] $C$ satisfies the following relative constraints: $C\cap\Tor(\partial P)$
is contained in the boundary $\partial{\mathbf Q}(+,+)$ of the positive quadrant $Q(+, +)$,
and each point 
$w\in C\cap\Tor(\partial P)$ is a smooth point of $C$ and $\Tor(\partial P)$, where they intersect with an even multiplicity $2k_w$.
\end{enumerate}  
More precisely, each invariant counts the pairs $(C, C_+)$ satisfying the following conditions
(see \cite{ISalg} for details):   
\begin{itemize}
\item the curve $C$ has given degree $2D_P$ and given genus $g \in \{0, 1, 2\}$; 
\item the half $C_+$ has given quantum index $\QI(C_+)$;
\item the curve $C$ is tangent, with given even intersection multiplicities $2k_w$, 
to the toric divisors at given smooth points $w$ of $\partial{\mathbf Q}(+, +)$, 
the points of tangencies being smooth for $C$ as well; 
the collection of fixed points of $\partial{\mathbf Q}(+, +)$ and the intersection multiplicities are chosen in such a way 
that, 
\begin{enumerate}
\item[(i)]
for each toric divisor $\Tor(\sigma)$, $\sigma\in P^1$, the sum of intersection multiplicities of $C$ and $\Tor(\sigma)$ 
at those of these points that belong to $\Tor(\sigma)$ is equal to the total intersection multiplicity 
$C \cdot \Tor(\sigma)$ 
of $C$ and $\Tor(\sigma)$;  
\item[(ii)] the Menelaus condition (see \cite{Mir}) is verified; 
\end{enumerate} 
\item the curve $C$ passes through given $g$ points in distinct non-positive open quadrants. 
\end{itemize}
Notice that these conditions imply that the pair $(C, C_+)$ is as in (a) and (b) above. 
The weights of the counted curves are $\pm1$: for one series of invariants, they are $W'(C,C_+)=(-1)^{f_1(C,C_+)}$, and for the other, $W''(C,C_+)=(-1)^{f_2(C,C_+)}$, where
\begin{align}f_1(C,C_+) = & \; e_+(C)+h_-(C)+\#\{w\in C\cap\Tor(\partial P)\ :\ k_w\ \text{even},\ \eps(C_+,w)=-1\}, \nonumber\\
f_2(C,C_+) = & \; e(C)+\#\{w\in C\cap\Tor(\partial P)\ :\ \eps(C_+,w)=-1\},\nonumber
\end{align}
with $e_+(C)$, respectively, $h_-(C)$, standing for the number of {\it elliptic nodes}
({\it i.e.}, real
double points that are transversal intersections of two complex conjugated local branches)
of $C$ in ${\mathbf Q}(+,+)$,
respectively, {\it hyperbolic nodes}
({\it i.e.}, real
double points that are transversal intersections of two real local branches)
of $C$
in $\Tor_\R(P) \setminus Q(+, +)$,
and $e(C)$ is the total number of elliptic nodes of $C$;
at last, $\eps(C_+,w)$ equals $1$ or $-1$ according as the orientation of $\partial{\mathbf Q}(+,+)$ is coherent or not with the complex orientation of $C_\R$ at the point $w$.
We 
show that these signs coincide up to a factor depending only on $g$, $D_P$,
and $\QI(C_+)$.

\begin{theorem}\label{th-signs} Under the above assumptions (a) and (b), one has 
\begin{equation}W''(C,C_+)=W'(C,C_+)(-1)^{
g+ p_a(D_P) + p_a(2D_P) + \frac{2D_P^2-\QI(C_+)}{4}},\label{eq-signs}\end{equation}
where $p_a(D)=\frac{D^2+DK}{2}+1$
is the arithmetic genus of any representative of a divisor class $D$,
and $K$ is the canonical class of $\Tor(P)$.
\end{theorem}

Before proving the theorem, we provide some information on the quantum index.
In particular, we show that the expression $\frac{2D_P^2-\QI(C_+)}{4}$
appearing in the righthand side of the formula (\ref{eq-signs})
is an integer (see Corollary \ref{cor-congruence}).

\subsection{Computation of quantum index}\label{sec-cqi}

The key role in the proof of Theorem \ref{th-signs} is played by Mikhalkin's formula for quantum index \cite[Theorem 3.4 and the next paragraph]{Mir}, which we present here in the form
\begin{equation}\QI(C_+)=-EN(C_+)-\Rot(C_\R)+\sum_{w\in C\cap\Tor(\partial P)}k_w\eps(C_+,w),\label{eq-mikh2}\end{equation}
where
\begin{itemize}\item $EN(C_+)$ is the intersection number of $C_+$ with $(\R^\times)^2$ equipped with the canonical orientation;
\item $\Rot(C_\R)$ is the sum of rotation numbers of real branches of $C$ equipped with the complex orientation and considered to be immersed into the corresponding (oriented) quadrant.
\end{itemize}

\begin{remark}\label{rem-signs}
In Mikhalkin's formula, there is the summand $-\frac{1}{2}\Rot_{\log}(C_\R)$, and the reason for an extra factor $\frac{1}{2}$ is that the logarithmic Gauss map target is $\PP^1_\R$, while in $\Rot(C_\R)$ we use the map to $S^1$.
\end{remark}



The canonical orientations of the non-positive quadrants induce uniquely defined orientation of non-positive segments of real toric divisors $\Tor_\R(\sigma)$, $\sigma\in P^1$. The sum of these oriented segments is an integral $1$-cycle $N$. We consider two $2$-cycles relative to $N$:
$${\mathcal C}_+=2{\mathbf Q}(+,+)+\sum_{\sigma\in P^1}(\Tor(\sigma)_--\Tor(\sigma)_+),$$
$${\mathcal C}=\sum_{\eps_1,\eps_2=\pm1}{\mathbf Q}(\eps_1,\eps_2)+\sum_{\sigma\in P^1}(\Tor(\sigma)_--\Tor(\sigma)_+),$$
where the hemisphere $\Tor(\sigma)_+$ induces the canonical orientation of the positive segment of $\Tor_\R(\sigma)$ for each $\sigma\in P^1$.

\begin{lemma}\label{lem-cyc}
The relative cycles ${\mathcal C}_+$ and ${\mathcal C}$ satisfy the following relations:
$$\partial{\mathcal C}_+=2N,\quad\partial{\mathcal C}=4N,$$
$$[{\mathcal C}-2{\mathcal C}_+]=0\in H_2(\Tor(P)).$$
\end{lemma}

{\bf Proof.}
Only the vanishing relation requires an explanation.
The vector space $H_2(\Tor(P);\Q)$ is generated by the classes of toric divisors, and we
show that the intersections of
$[{\mathcal C}-2{\mathcal C}_+]$
with all these classes vanish.

Observe that $[{\mathcal C}-2{\mathcal C}_+]\circ[2D_P]=0$. Indeed, the linear system $|2D_P|$ contains a smooth real curve $C$ with $C_\R=\emptyset$. Thus,
$$({\mathcal C}-2{\mathcal C}_+)\circ C=(\sum_{\sigma \in P^1}(\Tor(\sigma)_--\Tor(\sigma)_+)) \circ C = 0,$$ 
since the latter intersection consists of pairs of complex conjugate points with opposite intersection multiplicities.
Similarly, for any toric divisor $\Tor(\sigma)$ and a sufficiently large $n$,
the linear system $|2nD_P-2\cdot\Tor(\sigma)|$ contains a smooth real curve with an empty real part,
and hence
$$[{\mathcal C}-2{\mathcal C}_+]\circ[2nD_P-2\cdot\Tor(\sigma)] = 0.$$
Hence, we obtain $[{\mathcal C}-2{\mathcal C}_+]\circ[\Tor(\sigma)]=0$.
\proofend

\begin{lemma}\label{lem-sm}
Let
$(C,C_+)$
be an oriented real curve such that its singularities are ordinary nodes, its intersection points with $\Tor(\partial P)$ are smooth points of $C$ and $\Tor(\partial P)$, and all intersection multiplicities at the points of $C\cap\Tor(\partial P)$ are even. Then,
the following holds.

(1) There exists a small conjugation-invariant deformation of $C$ inside the linear system $|2D_P|$
such that
all intersection points with toric divisors are fixed {\rm (}intersection multiplicities at them being preserved{\rm )}, and 
all real nodes of $C$ are smoothed out: hyperbolic ones in accordance with the complex orientation of $C_\R$,
elliptic ones with appearance of small \emph{ovals}
{\rm (}{\it i.e.}, embedded circles bounding disks in the corresponding quadrants,
{\it cf.} \cite{Fied}{\rm )}.
The resulting deformed curve $C_{sm}$ is separating
as well, one of its halves $(C_{sm})_+$ is a small deformation of $C_+$, and
the real point set $(C_{sm})_\R$ of $C_{sm}$ is the union of disjoint ovals.

(2) Furthermore, there exists a small conjugation-invariant deformation of $C_{sm}$
inside the linear system $|2D_P|$ that preserves
the number of nodes
and slightly moves
all local real branches of $C_{sm}$ centered at the points $w\in (C_{sm})_\R\cap\Tor(\partial P)$
inside the corresponding quadrant
in such a way that the resulting deformed curve $C'_{sm}$ does not have real intersections
with $\Tor(\partial P)$.
\end{lemma}

{\bf Proof}.
The both statements are consequences of standard arguments of deformation theory ({\emph{cf}.,
for example, \cite[Proposition 4.4.3(a)]{GLS1}
and \cite[Theorem 1]{Sh1}).
\proofend

The following definition is largely inspired by the original constructions due to V. I. Arnold \cite{Arn} and V. A. Rokhlin \cite{Roh}.

\begin{definition}\label{def-arn}
An Arnold-Rokhlin surface $\Arn_{sm}(C_+)$
of $C_+$ is the integral $2$-cycle
combined of $(C'_{sm})_+$
and
the disks bounded by the ovals of $(C'_{sm})_\R$ in the corresponding quadrants;
for each disk, the orientation is chosen in such a way that it
induces on
the boundary of the disk the orientation opposite to the complex one.
\end{definition}

An oval of $C'_{sm}$ in a quadrant ${\mathbf Q}$ is
said to be \emph{negative} (respectively, \emph{positive})
if its complex orientation is (respectively, is not)
induced by the canonical orientation
of the disk surrounded by this oval in ${\mathbf Q}$, the canonical orientation
of the disk being coming from the orientation of $\mathbf Q$.
Denote by $\ell_+(C_+)$ and $\ell_-(C_+)$
(respectively, $\ell_+(C_+)\big|_{{\mathbf Q}(+,+)}$ and $\ell_-(C_+)\big|_{{\mathbf Q}(+,+)}$)
the numbers of positive and negative ovals
(respectively, positive and negative ovals in ${\mathbf Q}(+,+)$)
of $(C'_{sm})_+$.

\begin{proposition}\label{lem-int}
The quantum index $\QI(C_+)$ of $C_+$ satisfies the following relations:
\begin{align}\QI(C_+) & =
4\ell_+(C_+)\big|_{{\mathbf Q}(+,+)}-4\ell_-(C_+)\big|_{{\mathbf Q}(+,+)}
+ 2\sum_{w\in C\cap\Tor(\partial P)}k_w\eps(C_+,w), \nonumber \\
& = -2{\mathcal C}_+\circ\Arn_{sm}(C_+), \label{eq-qi1} \\
& \nonumber \\
\QI(C_+) & = \ell_+(C_+)-\ell_-(C_+)+\sum_{w\in C\cap\Tor(\partial P)}k_w\eps(C_+,w) \nonumber \\
& = -{\mathcal C}\circ\Arn_{sm}(C_+). \label{eq-qi2}
\end{align}
\end{proposition}
{\bf Proof.}
We start with establishing the second equalities in (\ref{eq-qi1}) and (\ref{eq-qi2}).
Let us slightly deform $\Arn_{sm}(C_+)$ into a
piecewise-smooth
cycle intersecting ${\mathcal C}_+$ and ${\mathcal C}$ transversally. 
\begin{itemize}\item First, for each oval $O$ of the curve $C'_{sm}$, take a vector field $\overline v_O$ on the disk $\delta_O$, surrounded by $O$, that is tangent to the oval $O$ in the direction of its complex orientation and has a unique non-degenerate singular point inside $\delta_O$.
\item Second,
we shift $\delta_O$ along the field $\bi\xi\overline v_O$, $\bi=\sqrt{-1}$, $0<\xi\ll1$,
so that $O$ moves inside $(C'_{sm})_+$.
\end{itemize}
The obtained $2$-cycle $\Arn'_{sm}(C_+)$ intersects the quadrants and the hemispheres $\Tor(\sigma)_{\pm}$, $\sigma\in P^1$, transversally with the following multiplicities:
\begin{itemize}
\item at each singular point of the vector field $\overline v_O$,
where the oval $O$ of $C'_{sm}$ lies in ${\mathbf Q}(\eps_1,\eps_2)$,
we get an intersection of $\Arn'_{sm}(C_+)$ with the quadrant ${\mathbf Q}(\eps_1,\eps_2)$
of multiplicity $1$ or $-1$ according as $O$ is positive or negative (see \cite{Roh});
\item given a point $w\in (C_{sm})_\R\cap\Tor(\sigma)$, $\sigma \in P^1$, the germ $(C_{sm},w)$
moves to a germ of $C'_{sm}$ intersecting $\Tor(\sigma)$ at $k_w$ pairs of complex conjugate points:
exactly $k_w$ of 
these intersection points belong to $(C'_{sm})_+$, and they belong to $\Tor(\sigma)_+$ or $\Tor(\sigma)_-$ 
according as $\eps(C_+,w)=1$ or $-1$; the intersection number of $\Arn'_{sm}(C_+)$
and $\Tor(\sigma)_--\Tor(\sigma)_+$ at such a point equals $-\eps(C_+,w)$.
\end{itemize}

To
finish the proof,
it remains to use Lemma \ref{lem-cyc}
and to notice that the first equality in (\ref{eq-qi2}) follows from Mikhalkin's formula (\ref{eq-mikh2}).
Indeed,
split the set of ovals of $C'_{sm}$ into two parts: the set ${\mathbf O}_h$
of ovals appeared from real branches of $C$ after smoothing out hyperbolic nodes and the set ${\mathbf O}_e$
of ovals appeared from elliptic nodes of $C$.
The contribution of ovals from ${\mathbf O}_h$ to the right-hand side of (\ref{eq-qi2}) coincides with $-\Rot(C_\R)$,
which is clear from the above computation of intersection multiplicities.
The contribution of ovals from ${\mathbf O}_e$ equals $EN(C_+)$.
To confirm this, consider the following simple model: an elliptic node $x^2+y^2=0$ at the origin and a positive oval $V=\{x^2+y^2=\xi^2\}$, $0<\xi\ll1$. The unit tangent vector to $V$ at the point $(-\xi,0)$ is $(0,0,1,0)$
in the coordinates $(\Rea x,\Ima x,\Rea y,\Ima y)$.
The vector directed towards $(C'_{sm})_+$ is then $(0,0,0,1)$. 
Hence, $C_+$ is given by $y=-x\bi$. The intersection number of $C_+$ with (canonically oriented) $\R^2$ equals
$$\det\left(\begin{matrix} 1 & 0 & 0 & -1 \\
0 & 1 & 1 & 0\\
1 & 0 & 0 & 0\\
0 & 0 & 1 & 0\end{matrix}\right)=-1.$$
\proofend

Using the first equality in (\ref{eq-qi1})
and calculations completely similar to the ones we have just made, we obtain the following statement.

\begin{lemma}\label{prop-qi}
One has
\begin{equation}\QI(C_+)=-4EN_+(C_+)-4\Rot_+(C_\R)+2\sum_{w\in C\cap\Tor(\partial P)}k_w\eps(C_+,w),\label{eq-mikh1}\end{equation}
where $EN_+(C_+)$
is the intersection number of $C_+$ with the
oriented quadrant ${\mathbf Q}(+, +)$, 
and $\Rot_+(C_\R)$ is the rotation number of the real branch in ${\mathbf Q}(+,+)$.
\proofend
\end{lemma}



\begin{corollary}\label{cor-congruence}
The quantity $(2D_P^2 - \QI(C_+))/4$ is an integer.
\end{corollary}

{\bf Proof}.
The statement follows from Lemma \ref{prop-qi}, the relation
$$\sum_{w\in C\cap\Tor(\partial P)}k_w\eps(C_+,w) = -D_PK - 2\sum_{\renewcommand{\arraystretch}{0.6}
\begin{array}{c}
\scriptstyle{w\in C\cap\Tor(\partial P)}\\
\scriptstyle{\eps(C_+,w)=-1}\end{array}}k_w,$$
and the fact that the numbers $D_P^2$ and $-D_PK$ have the same parity.
\proofend

\subsection{Proof of Theorem \ref{th-signs}}

The relation (\ref{eq-signs}) can be restated as the following congruence:
\begin{align}
& g - p_a(D_P) +p_a(2D_P) + \frac{2D^2_P - \QI(C_+)}{4} \nonumber \\
\equiv  \; & e_-(C)+h_-(C) + \#\{w\in C\cap\Tor(\partial P)\ :\ k_w\ \text{odd},\ \eps(C_+,w)=-1\}\mod2, \nonumber
\end{align} 
where
$e_-(C)$ is the number of elliptic nodes of $C$ outside of ${\mathbf Q}(+,+)$.
Using Lemma \ref{prop-qi}, the left-hand side of the above formula rewrites as
$$
\displaylines{
g - p_a(D_P) + p_a(2D_P) + E_+(C_+) + \Rot(C_\R)
+ \frac{1}{2}\left(D^2_P - \sum_{w\in C\cap\Tor(\partial P)}k_w\eps(C_+,w)\right) \cr
= g + p_a(2D_P) + E_+(C_+) + \Rot(C_\R)
+ \frac{1}{2}\left(-D_PK - \sum_{w\in C\cap\Tor(\partial P)}k_w\eps(C_+,w)\right) -1
}
$$
\begin{equation} = g + p_a(2D_P) + E_+(C_+) + \Rot(C_\R) + \sum_{\renewcommand{\arraystretch}{0.6}
\begin{array}{c}
\scriptstyle{w\in C\cap\Tor(\partial P)}\\
\scriptstyle{\eps(C_+,w)=-1}\end{array}}k_w - 1.\label{eq-signs2}\end{equation}
The following congruences hold:
$$E_+(C_+) \equiv e_+(C)\mod 2,$$
$$\Rot_+(C_\R) \equiv h_+(C)+1\mod 2,$$
$$\sum_{\renewcommand{\arraystretch}{0.6}
\begin{array}{c}
\scriptstyle{w\in C\cap\Tor(\partial P)}\\
\scriptstyle{\eps(C_+,w)=-1}\end{array}}k_w \equiv \#\{w\in C\cap\Tor(\partial P)\ :\ k_w\ \text{odd},\ \eps(C_+,w)=-1\}\mod 2,$$
$$e_-(C)+h_-(C) \equiv e_+(C)+h_+(C)+p_a(2D_P)+g\mod 2.$$
In particular, the second one comes from the observation that $C$ has one real branch in ${\mathbf Q}(+,+)$, and when we smooth out a hyperbolic node in accordance with the complex orientation (see \cite{Fied}) the rotation number is preserved, while the number of real branches changes by $\pm1$.
Combining the formula (\ref{eq-signs2}) with these congruences, we derive the required relation (\ref{eq-signs}).
\proofend

\section{Another definition of quantum index}\label{sec-another}

Following
our considerations in Section \ref{sec-cqi}, we suggest
another definition of the quantum index based on formula (\ref{eq-qi2}).
The new quantum index
is defined for oriented real curves on toric surfaces with the standard real structure
and on real del Pezzo surfaces with a non-empty real point set.
Given a real reduced curve $\newDelta$ 
representing the anti-canonical class
of the surface, the curves under consideration are supposed to satisfy the following conditions:
\begin{enumerate}\item[(I1)] all intersection points 
of the curves and $\newDelta$ are smooth points of the curves and of $\newDelta$;
\item[(I2)] the intersection multiplicities at real intersection points of the curves and $\newDelta$ are even;
\item[(I3)] the only singularities of the curves are ordinary nodes.
\end{enumerate}


\subsection{Homological meaning of quantum index}\label{sec-hom}
Let $X$ be a simply connected algebraic surface, $N\subset X$
an embedded piecewise-smooth circle such that $X\setminus N$ is smooth,
and $\Tub(N)$ a small open tubular neighborhood of $N$.
The Mayer-Vietoris homological sequence for the cover $X=(X\setminus N)\cup\Tub(N)$ yields an exact sequence
$$0\to H_2(\Tub(N)\setminus N)\simeq\Z\overset{\ina'_*}{\to} H_2(X\setminus N)\overset{\ina''_*}{\to} H_2(X)\to0.$$
Since $H_2(X)$ is free abelian, the sequence splits: $H_2(X\setminus N)\simeq H_2(X)\oplus\Z$.
Then the intersection number with a fixed class $\alpha\in H_2(X,N)$ such that $\alpha\circ \ina'_*(c)\ne0$, $c\in H_2(\Tub(N)\setminus N)$ a generator, defines a {\it quantum index homomorphism}
$$\QI_\alpha: H_1(X\setminus N)\to\Z,$$
which descents to injective maps $b+\ina'_*H_2(\Tub(N)\setminus N)\to\Z$ on each coset $b+\ina'_*H_2(\Tub(N)\setminus N)$. More precisely, we have isomorphisms
$$H_2(X\setminus N)\simeq H^2_{comp}(X\setminus N;\Z)\simeq H^2(X,N;\Z)\simeq \Hom(H_2(X,N),\Z),$$
and hence a non-degenerate pairing
$$H_2(X\setminus N)\otimes H_2(X,N)\to\Z.$$

\subsection{Quantum index for curves on toric surfaces}\label{sec-toric}
In the setting of Section \ref{sec-cqi}, given a toric surface $\Tor(P)$, we take $\alpha=[{\mathcal C}]$ in which case we have $\alpha\circ c=\pm4$. In \cite{Mir}, Mikhalkin restricts the quantum index function to oriented real curves
on toric surfaces that cross toric divisors in real and purely imaginary points (a corrected statement for the case of purely imaginary points is found in \cite{Bl}).

In fact, the quantum index is well-defined for the oriented real curves
$(C,C_+)\subset\Tor(P)$
that satisfy conditions (I1)-(I3).
Note that each real branch of such a curve is located in a certain closed quadrant and is contractible in this quadrant.
Another remark is that the construction of the deformed curves $C_{sm}$ and $C'_{sm}$ described in Lemma \ref{lem-sm},
works equally well for oriented real curves satisfying conditions (I1)-(I3). Thus,
we can define the $2$-chain $\Sigma$ formed by suitably oriented disks bounded in $\Tor(P)_\R\setminus N$ by the ovals of $C'_{sm}$.
Then, we
put
$$\Arn(C_+)=\Closure((C'_{sm})_+)+\Sigma$$ and define the quantum index $\QI(C_+)$ of $C_+$ as follows:
$$\QI(C_+)=-{\mathcal C}\circ\Arn(C_+).$$
This is a reformulation of the definitions of quantum index that appeared in \cite{Mir, Bl25}.

\subsection{Quantum index in a non-toric setup}\label{sec-nt}

Now, let $X$ be a real del Pezzo surface with a non-empty real part $X_\R$, and let $E\in|-K_X|$ be a smooth real (elliptic) curve with two real branches $B_1$, $B_2$.
Notice that that curve $E$ is separating,
and denote by $E_+$ and $E_-$
the connected components of $E\setminus E_\R$.
We equip $B_1$ and $B_2$ with the complex orientations induced from $E_+$
and understand $B_1+B_2$ as the oriented boundary of the closure of $E_+$.

Recall that each connected component of $X_\R$ is homeomorphic to either the sphere $S^2$, or the torus $(S^1)^2$,
or the connected sum $\#_s\PP^2_\R$, $1\le s\le 9$, of real projective planes  
(the classification of topological types of real parts of real del Pezzo surfaces is due to A. Commessatti \cite{Com};
see details in \cite[Theorems 6.11.11 and 17.3]{DIK}).
Since $E_\R$ represents the class dual to $w_1(X_\R)$, see \cite[Page 498]{BH}, the complement $X_\R\setminus E_\R$ is orientable. Denote by ${\mathbf S}(X,E)$ the set of connected components of $X_\R\setminus E_\R$.

In this section, for oriented real curves $(C,C_+) \subset X$
satisfying conditions (I1)-(I3) with respect to $E$, we define a generalized quantum index
\begin{equation}\widehat\QI_{X,E}(C_+)\in\Z\oplus\bigoplus_{M\in{\mathbf S}(X,E)}H_1(M; \Z).\label{e-gqi}\end{equation}
This quantum index can be used in definition of refined real enumerative invariants in the style of \cite{Mir} and \cite{ISalg}
for del Pezzo surfaces.

Recall that the construction of the deformed curves $C_{sm}$ and $C'_{sm}$, described in Lemma \ref{lem-sm}
for oriented real curves on toric surfaces, 
works equally well for oriented real curves $(C,C_+)\subset X$
satisfying conditions (I1)-(I3).

The first (integral-valued) coordinate $\QI_{X,E}(C_+)$ of $\widehat\QI_{X,E}(C_+)$ is
defined using formula (\ref{eq-qi2}),
where we consider only the left-most and the right-most sides,
and the non-degenerate pairing described in Section \ref{sec-hom}.
For such a definition, we
need
\begin{itemize}
\item to fix an appropriate $2$-cycle ${\mathcal C}$ relative to $B_1$,
\item and to construct in a canonical way absolute $2$-cycles $\Arn(C_+)$ for oriented real curves $(C,C_+)$
under consideration.
\end{itemize}
We accomplish
these two tasks in Sections \ref{sec-cycle} and \ref{sec-surfaces}, respectively.

The homological coordinate $\QI^{hom}_{X,E}(C_+)$ in (\ref{e-gqi}) is defined by
the real branches
of $C'_{sm}$ equipped with the complex orientation ({\it i.e.}, the one induced by the orientation of $(C'_{sm})_+$).

%


\subsubsection{Construction of a relative cycle ${\mathcal C}$}\label{sec-cycle}

Denote by $X^E_\R$ the union of the connected components of $X_\R$ that
have non-empty intersection with $E_\R$.
The following statement is an immediate consequence of \cite[Theorem 3.2]{Orevkov}.

\begin{lemma}\label{lem-cyc2}
The connected components of $X^E_\R\setminus E_\R$
can be oriented {\rm (}in a unique possible way{\rm })
so that their total boundary becomes equal to $2B_2-2B_1$.
\proofend
\end{lemma}

The connected components of $X^E_\R\setminus E_\R$ oriented as in Lemma \ref{lem-cyc2} naturally determine a $2$-chain
$\Sigma\subset X^E_\R$
and a $2$-cycle
$${\mathcal C}=\Sigma+E_--E_+$$
relative to $B_1$.
Note that $\partial{\mathcal C}=-4B_1$.


\subsubsection{Construction of Arnold-Rokhlin surfaces $\Arn(C_+)$}\label{sec-surfaces}

As a preparation for the construction of Arnold-Rokhlin surfaces, we extend the data $(X,E)$ with extra elements.
Let $M\in{\mathbf S}(X,E)$ be a connected component 
such that $H_1(M) \ne 0$
(the homology groups considered are with integer coefficients). 
The latter group is free abelian, and we can choose smooth oriented circles $O^M_i\subset M$, $1\le i\le r(M)$,
representing a basis of $H_1(M)$.
Since $X$ is simply connected, for each circle $O^M_i$,
there exists an immersed oriented disk $\delta^M_i\subset X\setminus B_1$ such that $\partial\delta^M_i=O^M_i$.
From now on, we fix a collection ${\mathcal S}=\{(O^M_i,\delta^M_i)\}_{M\in{\mathbf S}(X,E)}^{i=1,...,r(M)}$,
being referred to as {\it complementary data}. 

Let $(C,C_+)\subset X$ be an oriented real curve
satisfying conditions (I1)-(I3).
Pick a connected component $M \in {\mathbf S}(X,E)$. 
If $H_1(M) = 0$ ({\it i.e.}, $M$ is diffeomorphic to an open disk or a sphere), 
we take 
a $2$-chain $\Sigma_M$ assembled from 
disks in $M$ that are bounded by the real branches of $C'_{sm}$ 
in $M$
and oriented so that the induced orientations of the corresponding 
real branches are opposite to their complex orientation. 
Assume now that $H_1(M) \ne 0$.
For each (oriented) real branch $O \subset M$ of $C'_{sm}$, we perform the following procedure.
If $[O] = 0 \in H_1(M)$,
then $O$ bounds a uniquely determined component of $M \setminus O$,
and we attach the closure of this (appropriately oriented) component to $(C'_{sm})_+$.
If $[O] \ne 0 \in H_1(M)$,
then
$$[O]=\sum_{i=1}^{r(M)}n_i[O^M_i],$$
and the $1$-chain
$$O-\sum_{i=1}^{r(M)}n_iO^M_i$$
bounds a uniquely determined $2$-chain $\Sigma_O$ composed of
the closures of certain (appropriately oriented) components of the complement
$$M\setminus\left(O\cup\bigcup_{i=1}^{r(M)}O^M_i\right).$$
We attach to $(C'_{sm})_+$ the $2$-chain
$$-\Sigma_O+\sum_{i=1}^{r(M)}n_i\delta^M_i$$
along the real branch $O$.

Applying the procedure described above for all connected components $M \in {\mathbf S}(X, E)$ 
and all real branches $O \subset M$ of $C'_{sm}$, we obtain
a topological surface which is denoted by $\Arn_{X, E}(C_+)$,
or simply by $\Arn(C_+)$,
and is called an {\it Arnold-Rokhlin surface of $C_+$}.
This surface is said to be
{\it associated with the complementary data ${\mathcal S}$}
(though we do not specify $\mathcal S$ in the notation).

The integral-valued quantum index $\QI_{X, E}(C_+)$ of $C_+$ is defined by
$$\QI_{X,E}(C_+)=-{\mathcal C}\circ\Arn(C_+).$$
The natural question on the dependence of $\QI_{X,E}(C_+)$ of the choice of complementary 
data ${\mathcal S}$ (and choices made in the procedure above) is answered by the following self-evident lemma. 

\begin{lemma}\label{lem-comcol}
For a given 
$\theta\in\bigoplus_{M\in{\mathbf S}(X,E)}H_1(M)$, the difference $$\QI'_{X,E}(C_+)-\QI''_{X,E}(C_+)$$ between integral-valued quantum indices defined by two complementary data ${\mathcal S}'$ and ${\mathcal S}''$, does not depend on the choice
of an oriented real curve $(C,C_+)\subset X$ satisfying conditions (I1)-(I3) and such that
$\QI^{hom}_{X,E}(C_+)=\theta$.
\proofend
\end{lemma}

In Example \ref{ex-gqi} below, we show that, in general, the integral-valued and homological parts of
the generalized quantum index $\widehat\QI_{X,E}(C_+)$
do not determine each other.

\subsubsection{An important particular case}\label{sec-part}

Here, we specify certain real del Pezzo surfaces and oriented real curves
on them which admit a
simpler construction of Arnold-Rokhlin surfaces
and for which the integral-valued quantum index can be computed by 
an expression similar to the middle term of formula (\ref{eq-qi2}). 
These new Arnold-Rokhlin surfaces are denoted by $\Arn^0(C_+)$. 

Assume that
\begin{enumerate}\item[(A1)] each connected component of $X_\R$ is diffeomorphic either to $\PP^2_\R$, or to $S^2$
(for example, $X$ is the projective plane $\PP^2$ equipped with the standard real structure,
or a smooth two-component real cubic surface in $\PP^3$); 
\item[(A2)] an oriented real curve $(C,C_+)\subset X$ satisfies conditions (I1)-(I3) with respect to $E$.
\end{enumerate} 

Observe that under the above conditions, each real branch $O$ of $C'_{sm}$ bounds a disk $\delta_O$ in $X_\R$.
To obtain $\Arn^0(C_+)$, we attach these disks (sometimes slightly modified) to $(C'_{sm})_+$.
If $O$ is a real branch of $C'_{sm}$ such that $O\not\subset X^E_\R$,
then, $O$ lies in a spherical component $Y$ of $X_\R$, and we attach to $(C'_{sm})_+$
any of the two disks bounded by $O$ in $Y$.

Let a real branch $O$ of $C'_{sm}$ lie in a component $Y$ of $X^E_\R$.
Suppose that $Y\simeq\PP^2_\R$.
Then, $Y$ contains a real branch $B_i$ of $E$ ($i\in\{1,2\}$) which is not contractible in $Y$, and the complement $Y\setminus B_i$ is diffeomorphic to an open disk. Thus, $O$ bounds a disk $\delta_O\subset Y\setminus B_i$. If $B_1\cap\delta_O=\emptyset$, we attach the disk $\delta_O$ to $(C'_{sm})_+$. If $i=2$ and $B_1\subset\delta_O$ (in such a case $B_1$ bounds a disk $\delta_1\subset\delta_O$), we attach to $(C'_{sm})_+$ the disk $\delta_O$ and perform the following modification in order to separate it from $B_1$ (and from $E$). 
Choose a smooth vector field $\overline v$ on $\delta_O$ which is identically zero outside a small open tubular neighborhood
$\Tub(B_1)\subset Y$, does not vanish in $\Tub(B_1)$, is transversal to $B_1$, and satisfies there $|\overline v|\ll1$.
Then, we shift the disk $\delta_O$ along the vector field $\bi\overline v$.

Suppose that $Y\simeq S^2$. If $B_1\cap Y=\emptyset$ and $B_2\subset Y$, we attach to $(C'_{sm})_+$ the disk bounded by $O$ in $Y\setminus B_2$. If $B_1\subset Y$, we attach to $(C'_{sm})_+$ the disk bounded by $O$ in $Y\setminus B_1$.

The following lemma compares the integral-valued quantum indices defined by $\Arn^0(C_+)$ and by $\Arn(C_+)$ constructed in Section \ref{sec-surfaces}.

\begin{lemma}\label{lem-gqi0}
Under assumptions (A1) and (A2), for an appropriate choice of the complementary data, one has $${\mathcal C}\circ\Arn(C_+)={\mathcal C}\circ\Arn^0(C_+)$$ for all oriented real curves $(C,C_+)\subset X$ under consideration.
\end{lemma}

{\bf Proof.}
Under condition (A1), either all components of $X^E_\R\setminus E_\R$ are open disks, or all but one components are open disks, while the remaining component $M$ is an open annulus. In the former case, $\Arn(C_+)=\Arn^0(C_+)$ by construction. In the latter case, we have to consider only real branches of $C'_{sm}$ representing generators of $H_1(M)\simeq\Z$. Let us take the following complementary data:
an oriented smooth circle
$O^M_1$
in $M$ representing a generator of $H_1(M)$,
and the disk $\delta^M_1$ with the boundary $O^M_1$ which is chosen along the rules presented before the lemma for real branches of $C'_{sm}$. It is easy to see that the difference between $\Arn(C_+)$ and $\Arn^0(C_+)$ is a $2$-cycle contractible in $X_\R\setminus B_1$.
\proofend

We give a formula for computing the integral-valued quantum index $-{\mathcal C}\circ\Arn^0(C_+)$ under some extra assumption. 
For an oriented real curve $(C,C_+)$
satisfying conditions (I1)-(I3), a real branch $O\subset X^E_\R$ of $C'_{sm}$
is said to be {\it negative} (respectively, {\it positive}) if
\begin{itemize}
\item either $\delta_O$ is contained in a connected component of $X_\R \setminus E_\R$
and the complex orientation of $O$ is (respectively, is not)
induced by the canonical orientation of $\delta_O$ (see Lemma \ref{lem-cyc2});
\item or $\delta_O$
contains $B_i$ ($i=1,2$), and the orientation of $\delta_O$ induced by the complex orientation of $O$ is (respectively, is not)
induced by the canonical orientation of the disk $\delta_i$ surrounded by $B_i$.
\end{itemize}
Furthermore, for each point $w\in C_\R\cap E_\R$, 
we define $\eps(C_+,w)=\pm1$ as follows: we set $\eps(C_+,w)=1$ if 
the complex orientations of $C_\R$ and $E_\R$ at $w$ are coherent,
otherwise we set $\eps(C_+,w)=-1$ (cf. Section \ref{sec-main}). 


\begin{proposition}\label{prop-generalized_qi}
Assume that $C\cap E\subset E_\R$. Then,
$$
-{\mathcal C}\circ\Arn^0(C_+) = \ell_+(C_+)-\ell_-(C_+)
+ \sum_{w \in C \cap E_\R}\eps(C_+,w)\frac{(C\cdot E)_w}{2},
$$
where $\ell_+(C_+)$ and $\ell_-(C_+)$ are, respectively, the numbers of
positive and negative real branches of $C'_{sm}$. 
\end{proposition}

{\bf Proof}. The arguments are completely similar to the ones
used in the proof of the equality (\ref{eq-qi2}) in Proposition \ref{lem-int}. We skip the details.
\proofend

\begin{example}\label{ex-gqi}
{\rm We construct a series of oriented real curves in the projective plane 
showing 
that the integral-valued and homological parts of the generalized quantum index $\widehat\QI_{X,E}(C_+)$
(see formula (\ref{e-gqi})) 
in general
do not determine each other.
Let $X=\PP^2$, and let $E$ be a smooth real cubic curve with a non-contractible (in $\PP^2_\R$)
real branch $B_1$ and an oval $B_2$. Consider a real conic $C_2$ (respectively, $C'_2$) such that
its real part surrounds (respectively, does not surround) the oval $B_2$, 
and $C_2$ (respectively, $C'_2$)
is quadratically tangent to $B_1$ 
at 
two points and to $B_2$ at one point. 
By Proposition \ref{prop-generalized_qi}, 
the quantum index $\QI_{X,E}((C_2)_+)$ is 
equal to $\pm 2$ depending on the chosen half $(C_2)_+$ of 
$C_2$. 
The same statement is true for the halves of $C'_2$.
Notice that $(C_2)_\R$ represents a generator of $H_1(M)\simeq\Z$, where $M$ is the annulus between $B_1$ and $B_2$,
and the class of $(C'_2)_\R$ is trivial in $H_1(M)$. 
Denote by $U_k$ (respectively, $U'_{k'}$) a union of $k \geq 1$ 
(respectively, $k' \geq 1$) distinct real conics obtained from $C_2$ (respectively, $C'_2$)
by small moves of its tangency points to $B_1$ and $B_2$.
For any class $\eta \in H_1(M)$ and any even integer $\iota$, 
one can choose positive integers $k$, $k'$ and an appropriate perturbation 
of the union $U_k \cup U'_{k'}$ in order to obtain an oriented real curve $(C, C_+)$ of degree $2(k + k')$
such that 
\begin{itemize}
\item $C$ is tangent to $B_1$ in $2(k+k')$ points and to $B_2$ in $k+k'$ points;
\item $(C, C_+)$ has 
$\iota$ as the integral-valued quantum index
and 
$\eta$ as the homological part of the generalized quantum index.
\end{itemize} 
} 
\end{example}

{\ncsc Sorbonne Universit\'e and Universit\'e Paris Cit\'e, CNRS, IMJ-PRG \\[-21pt]

F-75005 Paris, France} \\[-21pt]

{\it E-mail address}: {\ntt     ilia.itenberg@imj-prg.fr}

\vskip10pt

{\ncsc School of Mathematical Sciences \\[-21pt]

Raymond and Beverly Sackler Faculty of Exact Sciences\\[-21pt]

Tel Aviv University \\[-21pt]

Ramat Aviv, 6997801 Tel Aviv, Israel} \\[-21pt]

{\it E-mail address}: {\ntt shustin@tauex.tau.ac.il}

\end{document}